 \documentclass[12pt]{amsart}

\usepackage{amssymb,latexsym}
\usepackage{amsmath,amsfonts,amsthm,amscd,amsxtra,enumerate}
\usepackage{mathtools}
\usepackage{epsfig, hyperref}
\usepackage[margin=1in]{geometry}
\usepackage[utf8]{inputenc}
\usepackage{graphicx}
\usepackage{tikz, wrapfig,array}

\newtheorem{con}{ Conjecture}[section]
\newtheorem{definition}[con]{ Definition}

\newtheorem{example}[con]{ Example}
\newtheorem{theorem}[con]{ Theorem}
\newtheorem{lemma}[con]{ Lemma}
\newtheorem{remark}[con]{ Remark}
\newtheorem{corollary}[con]{ Corollary}
\newtheorem{proposition}[con]{ Proposition}

\newcommand{\sk}{{\ensuremath{\sf k }}}
\newcommand{\m}{\ensuremath{\mathfrak m}}

\newcommand{\R}{\mathcal{R}}
\newcommand{\p}{\mathfrak{p}}
\newcommand{\ass}{\operatorname{Ass}}

\DeclareMathOperator{\supp}{supp}

\providecommand\reg{\text{\rm reg}}

\providecommand{\BN}{{\mathbb N}}
\providecommand{\BZ}{{\mathbb Z}}

\begin{document}

\title{Regularity of Symbolic Powers of Edge Ideals}

\author{A. V. Jayanthan}
\email{jayanav@iitm.ac.in}
\address{Department of Mathematics, I.I.T. Madras, Chennai, 600036,
INDIA.}

\author[Rajiv Kumar]{Rajiv Kumar}
\email{gargrajiv00@gmail.com}
\address{Department of Mathematics, The LNM Institute of Information
Technology, Jaipur, Rajasthan, 302031, INDIA.}




\title{Regularity of Symbolic Powers of Edge Ideals}

\begin{abstract}
In this article, we prove that for several classes of graphs, the
Castelnuovo-Mumford regularity of symbolic powers of their edge ideals
coincide with that of their ordinary powers.
\end{abstract}

\maketitle
\section{Introduction}

This article is motivated by the results in the paper \cite{GHRS}. Gu
et al. in \cite{GHRS} studied the properties and invariants associated
with symbolic powers of edge ideals of
unicyclic graphs. Let $G$ be a finite simple graph on the vertex set
$x_1, \ldots, x_n$ and $I(G)$ denote the ideal in the polynomial ring
$S = \sk[x_1, \ldots, x_n]$ generated by $\{x_ix_j \mid \{x_i, x_j\}
\text{ is an edge of } G\}$, where $\sk$ is a field. There have been
a lot of research on connection between algebraic properties of
$I(G)^s$ with the combinatorial properties of $G$, see \cite{BBH17} and
the references there in. In the geometrical context, the symbolic
powers have more importance since it captures all polynomials that
vanishes with a given multiplicity. Algebraically, the symbolic powers
are harder to compute or handle. In our situation, we can observe that
$I(G)^{(s)} = \bigcap_{\p \in \ass(I)}\p^s.$ It was proved by Simis,
Vasconcelos and Villarreal that $G$ is bipartite if and only if
$I(G)^{(s)} = I(G)^s$ for all $s \geq 1$, \cite{svv94}. It has
been conjectured by N. C. Minh that if $G$ is a finite simple graph,
then $\reg(I(G)^{(s)}) = \reg(I(G)^s)$ for all $s \geq 1$, see
\cite{GHRS}. Gu et al., in \cite{GHRS}, proved this conjecture for odd
cycles. Recently, the conjecture has been proved for the classes of
unicyclic graphs, chordal graphs and Cameron-Walker graphs by Seyed
Fakhari, \cite{SA3, SA2, SA1}. In \cite{KS19}, Kumar and Selvaraja
generalized a result of Seyed Fakhari to prove Minh's conjecture for a
class of graphs obtained by attaching complete graphs to vertices of
unicyclic graphs.

In this article, we extend some of the results in \cite{GHRS} to
prove the equality of regularity of ordinary powers with that of
symbolic powers for certain classes of graphs. Our main theorem is
stated as follows:
\vskip 2mm \noindent
\textbf{Theorem} \ref{mainTheorem}\textbf{.}
{\em Let $G$ be a graph obtained by taking clique sum of a $C_{2n+1}$ and
some bipartite graphs. Let $H$ be an
induced subgraph of $G$ on vertices $V\setminus\bigcup_{x\in
V(C_{2n+1})}N_G(x)$. Assume that none of the vertices of $H$ is part
of any cycle in $G$. If $\nu(G)-\nu(H)\geq 3$, then
$\reg\left(I^{(s)}\right)=\reg\left(I^s\right)$.}

As in \cite{GHRS}, the approach is through understanding the symbolic
power as a sum of product of ordinary powers of certain related
ideals. We use this decomposition to study the regularity of symbolic
powers of edge ideals of graphs whose each odd cycle is a dominant
odd cycle.
\vskip 2mm \noindent
\textbf{Theorem} \ref{Theorem}\textbf{.}
{\em Let $G'$ be a clique sum of $r$ cycles of size $2n+1$, say $C_1,\dots,
C_r$, and $G$ be a graph by taking the clique sum of $G'$ and some
bipartite graphs. If $N_G(C_i)=V(G)$ for any odd
cycle $C_i$ in $G$, then $\reg\left(I^{(s)}\right)=\reg\left(I^{s}\right)$ for all $s\geq 1$.
}

The article is organized as follows. We collect the required
terminologies and results in Section 2. In Section 3, we obtain the
decomposition for symbolic powers in terms of ordinary powers and use
it to prove Theorem \ref{Theorem}. In the final section, we prove
Theorem \ref{mainTheorem}.

\noindent
\textbf{Acknowledgement:} We thank Yan Gu for going through a
preliminary version of the article and making some valuable comments.
We are also thankful to the anonymous reviewer for reading the
manuscript carefully and asking pertinent questions which lead us to
finding a gap in one of the lemmas in the first version.

\section{Preliminaries}
Throughout this paper, all graphs considered are assumed to be finite
and simple. For a graph $G$ with vertex set $V(G) = \{x_1, \ldots, x_n\}$,
$S$ denotes the polynomial ring $\sk[x_1, \ldots, x_n]$ and $\m$
denotes the unique graded maximal ideal in $S$.
In this section, we recall the definitions and results that are needed
for the rest of the paper. We begin by recalling the some of the
terminologies related to finite simple graphs.
\begin{definition}{\rm Let $G$ be a graph on the vertex set $V$. Then,
\begin{enumerate}[i)]
\item set $\alpha(G):=\min\{|C|: C \text{ is a vertex cover of } G\}$;
\item the graph $G$ is called \emph{decomposable} if there
exists a partition of $V=V_1\sqcup\dots\sqcup V_r$ such that
$\sum \alpha(G_i)=\alpha(G)$, where $G_i$ is induced
subgraph of $G$ on $V_i$. If $G$ is not decomposable, then
$G$ is called \emph{indecomposable};
\item for $T \subset V$, $G \setminus T$ denote the induced
subgraph of $G$ on the vertex set $V \setminus T$;
\end{enumerate}
}\end{definition}

It was shown by Harary and Plummer, \cite{HP} that every
indecomposable contains an odd cycle.  We now recall the duplication
and parallelization.
\begin{definition}{\rm Let $G$ be a graph on $n$ vertices and ${\bf
  v}=(v_1,\dots, v_{n})\in\BN^{n}.$
	\begin{enumerate}[\rm i)]
		\item The \emph{duplication} of a vertex $x$ of $G$ is the graph obtained from $G$ by adding a vertex $x'$ and all edges $\{x',y\}$ for all $y\in N_G(x)$.
		\item The \emph{parallelization} of $G$ with respect to ${\bf
		  v}$, denoted by $G^{\bf v}$, is the graph obtained from $G$
		  by deleting $x_i$ if $v_i= 0$ and duplicating $v_i-1$ times
		  $x_i$ if $v_i\geq 1$.
	\end{enumerate}
}\end{definition}
For an ideal $I$ in a commutative ring $A$, let $\R_s(I) :=
\oplus_{n\geq 0} I^{(n)}t^n$ denote the symbolic Rees algebra of $I$.
For a vector $\mathbf{v} \in \BN^n$, let $\mathbf{x}^{\mathbf{v}}$ be
the monomial $x_1^{v_1}\cdots x_n^{v_n} \in \sk[x_1, \ldots, x_n]$.
Mart\'{i}nez-Bernal et al. obtained the $\sk$-algebra generators for
the symbolic Rees algebra:

\begin{theorem}\cite[Lemma 2.1]{BRC2011}\label{implosiveThm}
	Let  $G$ be a graph on $V$.
	Then $$\R_s(I)=\sk[x^{\bf v}t^b: G^{\bf v} \text{ is an
	indecomposable graph, } {\bf v}\in \BN^{|V|} \text{ and } b=\alpha(G^{\bf v})].$$
\end{theorem}
Here we recall the definition of implosive graphs.
\begin{definition}\hfill{}{\rm\
	\begin{enumerate}[i)]
		\item A graph $G$ is called \emph{implosive} if symbolic Rees algebra of $I$ is generated by monomials of the form $x^{\bf v}t^b$, where ${\bf v}=\{0,1\}^{|V|}$.
		\item Let $G_1$ and $G_2$ be graphs. Suppose $G_1\cap G_2=K_r$ is a complete graph, where $G_1\neq K_r$ and $G_2\neq K_r$. Then $G_1\cup G_2$ is called the \emph{clique-sum} of $G_1$ and $G_2$. 
	\end{enumerate}
}\end{definition}

\begin{remark}\cite[Theorem 2.3, Theorem 2.5]{FGR}\label{implosiveRem}\hfil{}{\rm
	\begin{enumerate}[i)]
		\item If $G$ is a cycle, then $G$ is implosive.
		\item The clique-sum of implosive graphs is implosive. 
	\end{enumerate}
}\end{remark}

\section{Regularity of Dominant Cycles}
Gu et al. in \cite{GHRS} shows that if $G$ is unicyclic graph with
$C_{2n+1}=(x_1,\dots, x_{2n+1})$, then  $I^{(s)}=\sum\limits_{i=0}^{k}
I^{s-i(n+1)}(x_1\cdots x_{2n+1})^i$, where $s=k(n+1)+r$ for some $k\in
\BZ$ and $0\leq r\leq n$. In this section, we generalize some of the
results in sections 3 and 5 of \cite{GHRS} and use it to compute the
regularity of the symbolic powers, generalizing \cite[Theorem 5.3]{GHRS}.

\begin{lemma}\label{implLemma}
Let $G'$ be a clique sum of $r$ cycles of size $2n+1$, say $C_1,\dots,
C_r$, and $G$ be a graph by taking the clique sum of $G'$ and some
bipartite graphs. Let $I = I(G)$ and $J=( u_{C_1},\dots, u_{C_r}) $,
where $u_{C_i} = \prod_{j=1}^{2n+1}x_{i_j}$, the product of variables
corresponding to the vertices of the cycle $C_i$. Then $I^{(s)}=I^s$
for all $s\leq n$ and $I^{(s)}=\sum\limits_{i=0}^{k} J^iI^{s-i(n+1)}$,
where $s=k(n+1)+r$ for some $k\in \BZ$ and $0\leq r\leq n$.
\end{lemma}
\begin{proof}
Since $G$ is the clique sum of odd cycles and bipartite graphs, by
\cite[Theorems 2.3, 2.5]{FGR}, we get that $G$ is implosive.  By
\cite[Theorem 2]{HP}, any indecomposable induced subgraph of $G$ is
contained in $C_i$ for some $i$ or an edge. Moreover, by
\cite[Corollary 1b]{HP}, an indecomposable induced subgraph of $C_i$
is either itself or an edge. Hence by Theorem \ref{implosiveThm}, we
get $\mathcal{R}_s(I)=S[It, Jt^{n+1}]$. Now comparing the graded
components on both sides of the above equality, we get $I^{(s)}=I^s$
for all $s\leq n$ and $I^{(s)}=\sum\limits_{i=0}^{k} J^iI^{s-i(n+1)}$. 
%
\end{proof} 

To study the regularity of $I^{(s)}$, we need to understand the
structure of $I^{(s)} \cap \m^{2s}$. This is done by studying the
intersection with each of the term appearing in the summation in the
previous result.

\begin{lemma}\label{interLem}
	Let $G$ be a graph as in Lemma \ref{implLemma}. Then $$I^{(s)}\cap\mathfrak{m}^{2s}=\sum\limits_{i=0}^{k}J^i\mathfrak{m}^iI^{s-i(n+1)}.$$
\end{lemma}

\begin{proof} By Lemma \ref{implLemma}, it is enough to show that
  $ J^iI^{s-i(n+1)}\cap \mathfrak{m}^{2s} = J^i\mathfrak{m}^iI^{s-i(n+1)}.$
Since $J^i\mathfrak{m}^iI^{s-i(n+1)}\subset J^iI^{s-i(n+1)}$ and $J^i\mathfrak{m}^iI^{s-i(n+1)}\subset \mathfrak{m}^{2s}$, we get $$J^i\mathfrak{m}^iI^{s-i(n+1)}\subset J^iI^{s-i(n+1)}\cap \mathfrak{m}^{2s}.$$
	
For the reverse containment, let $u\in
J^iI^{s-i(n+1)}\cap\mathfrak{m}^{2s}$. Write $u=fgh$, where $f\in
G(J^i)$, $g\in G(I^{s-i(n+1)})$. Note that $u\in
\mathfrak{m}^{2s}$implies that $\deg(u)\geq 2s$. Since $\deg(f)=i(2n+1)$ and $\deg(g)=2s-2i(n+1)$, we get that $\deg(h)\geq i$ which completes the proof.
\end{proof}

As an immediate consequence, we obtain the intersection $I^{(s)} \cap
\m^{2s}$ for the class of graphs that we are considering.

\begin{corollary}
Let $G$ be a graph as in Lemma \ref{implLemma}. If $N_G(C_i)=V$ for any odd cycle
$C_i$ in $G$,   then $I^{(s)}\cap\mathfrak{m}^{2s}=I^s$. 
\end{corollary}
\begin{proof}
We show that $\m J \subseteq I^{n+1}$. Let $x_i \in V(G)$ and
$u_{C_i} = \prod_{j=1}^{2n+1} x_{i_j}$ be a minimal generator of $J$.
Without loss of generality, let $x_{i_1} \in N_{C_i}(x_i)$. Then
$x_iu_{C_i} = x_ix_{i_1}\cdot x_{i_2}x_{i_3} \cdots
x_{i_{2n}}x_{i_{2n+1}} \in I^{n+1}.$  Hence $\m J \subseteq I^{n+1}$
so that $\m^iJ^i \subseteq I^{i(n+1)}$.
\end{proof}

For a homogeneous ideal $I \subset S$, let $\alpha(I)$ denote the
least degree of a minimal generator of $I$. The \textit{Waldschmidt}
constant of $I$ is defined to be $\hat{\alpha}(I) :=
\displaystyle{\underset{s \to \infty}{\lim}
\frac{\alpha(I^{(s)})}{s}}$. The real number $\rho(I) = \sup\{s/t \mid
I^{(s)} \not\subset I^t\}$ is called \textit{resurgence} number of
$I$ and $\rho_a(I) = \sup\{s/t \mid I^{(sr)} \not\subset I^{tr}
\text{ for all } r \gg 0 \}$ is called \textit{asymptotic
resurgence} number of $I$. We compute the Waldschmidt constant,
resurgence and asymptotic resurgence number of the edge ideals of the
graphs considered in Lemma \ref{implLemma}.
\begin{corollary}
  Let $G$ be as in Lemma \ref{implLemma}. Then
  \begin{enumerate}
	\item $\alpha(I(G)^{(s)}) = 2s - \lfloor \frac{s}{n+1} \rfloor$ for
	  all $s \in \mathbb{N}$;
	\item $\hat{\alpha}(I(G)) = \frac{2n+1}{n+1}$;

	\item $\alpha(I(G)^{(s)}) < \alpha(I^t)$ if and only if
	  $I(G)^{(s)} \not\subset I^t$;
	\item $\rho(I(G)) = \rho_a(I(G)) = \frac{2n+2}{2n+1}$.
  \end{enumerate}
\end{corollary}
\begin{proof}
Since the proof is exactly same as the proof of
\cite[Theorem 3.6]{GHRS}, we skip it here.
\end{proof}

We now generalize \cite[Theorem 5.3]{GHRS}.

\begin{theorem}\label{Theorem}
Let $G$ be as in Lemma \ref{implLemma}. If $N_G(C_i)=V$ for any odd
cycle $C_i$ in $G$, then $\reg\left(I^{(s)}\right)=\reg\left(I^{s}\right)$ for all $s\geq 1$.
\end{theorem}
\begin{proof}
Suppose $\nu(G)=1$. Then $G$ is either $C_5$ or is the clique-sum of a
$C_3$, say $T$, with several copies of $C_3$, say $G_1,
\ldots, G_r$ along the edges of $T$ and copies of $P_2$, say $G_{r+1}, \ldots, G_s$
along the vertices of $T$. If $G = C_5$, then the assertion is proved
in \cite{GHRS}. If $G \neq C_5$, then $G^c$ is the clique-sum of a $K_s$
with $s-r$ copies of $C_3$ along the edges of $K_s$ and with $r$ many
edges along the vertices of $K_s$. Hence $G$ is a co-chordal graph.
Therefore, $S/I^s$ has linear resolution for all $s \geq 1$.
Consider the short exact sequence
\begin{eqnarray}\label{ses1}
0\longrightarrow \dfrac{S}{I^s}\longrightarrow\dfrac{S}{I^{(s)}}\oplus
\dfrac{S}{\mathfrak{m}^{2s}}\longrightarrow
\dfrac{S}{I^{(s)}+\mathfrak{m}^{2s}}\longrightarrow 0.
\end{eqnarray}
Note that since $I^{(s)}$ contains a minimal generator of degree $2s$,
$\reg(S/I^{(s)}) \geq 2s - 1 = \reg(S/\m^{2s})$. Also, $S/(I^{(s)}+
\m^{2s})$ is Artinian, $[S/(I^{(s)} + \m^{2s})]_{2s-1} \neq 0$ and 
$[S/(I^{(s)}+\m^{2s})]_{2s} = 0$ so that $\reg(S/I^{(s)}+\m^{2s}) =
2s-1$.  Hence it follows from the exact sequence (\ref{ses1}) that
$\reg(S/I^{(s)}) \leq 2s - 1$. Therefore $\reg(S/I^{(s)}) = 2s-1 =
\reg(S/I^s)$.

Assume now that $\nu(G)\geq 2$.
Since $S/(I^{(s)} + \mathfrak{m}^{2s})$ is Artinian, the regularity is
given by the socle degree. Hence
$\reg\left(\dfrac{S}{I^{(s)}+\mathfrak{m}^{2s}}\right)=2s-1=\reg\left(\dfrac{S}{\mathfrak{m}^{2s}}\right)$
and by \cite[Theorem 4.6]{GHRS} $\reg\left(\dfrac{S}{I^{(s)}}\right)\geq 2s+\nu(G)-2$. Since
$\nu(G)\geq 2$, this implies that
$\reg\left(\dfrac{S}{I^{(s)}}\right)>\reg\left(\dfrac{S}{I^{(s)}+\mathfrak{m}^{2s}}\right)$.
Hence it follows from the short exact sequence (\ref{ses1}) that $\reg\left(I^{(s)}\right)=\reg\left(I^s\right).$
\end{proof}

\begin{example}{\em
We would like to note here that the class of graphs that we have considered 
here is more general than unicyclic graphs with a
dominating	odd cycle which are considered in \cite{GHRS}. 

\begin{minipage}{\linewidth}
\begin{minipage}{0.53\linewidth}
For example, the graphs given on the right are not unicyclic graphs but
satisfy the hypotheses of Theorem \ref{Theorem}. The first one is a
clique sum of $C_5$ with some bipartite graphs which contain cycles.
The second graph on the right is a clique sum of three $C_3$'s.

\noindent
\end{minipage}
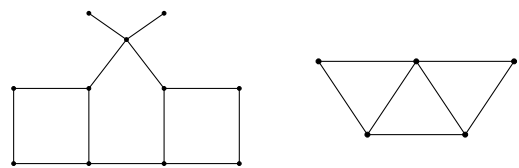
	\begin{minipage}{0.25\linewidth}
\[
		\begin{tikzpicture}[scale=0.5]
		\tikzstyle{edge} = [draw,thick,-]		
		\draw[-] (0,0) -- (6,0)--(6,2)--(4,2)--(4,0);
		\draw[-] (2,0) -- (2,2)--(0,2)--(0,0);
		\draw[-] (2,2)--(3,3.3);
		\draw[-] (2,4)--(3,3.3)--(4,4);
		\draw[-] (4,2)--(3,3.3);
		\draw [fill] (0,0) circle [radius=0.05];
		\draw [fill] (2,0) circle [radius=0.05];		
		\draw [fill] (4,0) circle [radius=0.05];		
		\draw [fill] (6,0) circle [radius=0.05];	
		\draw [fill] (0,2) circle [radius=0.05];
		\draw [fill] (2,2) circle [radius=0.05];		
		\draw [fill] (4,2) circle [radius=0.05];		
		\draw [fill] (6,2) circle [radius=0.05];
		\draw [fill] (3,3.3) circle [radius=0.05];
		\draw [fill] (2,4) circle [radius=0.05];		
		\draw [fill] (4,4) circle [radius=0.05];
				
		\node [below] at (0,0) { };
		\node [below] at (2,0) { };		
		\node [below] at (4,0) { };		
		\node [below] at (6,0) { };
		\node [above] at (0,2) { };
		\node [above] at (2,2) { };		
		\node [above] at (4,2) { };		
		\node [above] at (6,2) { };	
		\node [above] at (3,3.3) { };
		\node [above] at (1,3.3) { };		
		\node [above] at (5,3.3) { };
		\end{tikzpicture}
	\]
	\end{minipage}
	\begin{minipage}{0.2\linewidth}
		\[
		\begin{tikzpicture}[scale=0.65]
		\tikzstyle{edge} = [draw,thick,-]		
		\draw[-] (0,0) -- (2,0)--(3,1.5)--(1,1.5)--(-1,1.5)--(0,0);
		\draw[-] (2,0)--(1,1.5) --(0,0);

		\draw [fill] (0,0) circle [radius=0.05];
		\draw [fill] (2,0) circle [radius=0.05];		
		\draw [fill] (3,1.5) circle [radius=0.05];		
		\draw [fill] (1,1.5) circle [radius=0.05];	
		\draw [fill] (-1,1.5) circle [radius=0.05];
		
		\node [below] at (0,0) { };
		\node [below] at (2,0) { };		
		\node [above] at (3,1.5) { };		
		\node [above] at (1,1.5) { };
		\node [above] at (-1,1.5) {};
		\end{tikzpicture}
		\]	
	\end{minipage}
\end{minipage}	
}
\end{example}

\section{Regularity of Unicyclic Graphs}
In this section, we focus on graphs which has only one odd cycle. For
the rest of the paper, let $G$ be a graph obtained by taking
clique-sum along the vertices or edges of an odd cycle $C_{2n+1}$ and
some bipartite graphs. Let $V(C_{2n+1}) = \{x_1, \ldots,
x_{2n+1}\}, ~N_G(C_{2n+1}) \setminus V(C_{2n+1}) = \{y_1, \ldots,
y_l\}$ and $V(G) \setminus N_G(C_{2n+1}) = \{z_1, \ldots, z_m\}$.
%
%
Now we set $I = I(G), ~\mu=x_1\cdots x_{2n+1},$
$L=( x_1,\dots ,x_{2n+1}, y_1,\dots, y_l),
K=( z_1,\dots, z_m)$ and $\m$ the homogeneous maximal ideal in
$\sk[L,K]$. For any monomial ideal $J$, let
$G(J)$ denote the set of minimal monomial generators of $J$.
We first give a refinement of the decomposition of $I^{(s)} \cap
\m^{2s}$.
\begin{lemma}
	$I^{(s)}\cap \mathfrak{m}^{2s}=\sum\limits_{i=0}^{k}\mu^iK^iI^{s-i(n+1)}$.
\end{lemma}

\begin{proof}
Using Lemma \ref{interLem}, we get $I^{(s)}\cap
\mathfrak{m}^{2s}=\sum\limits_{i=0}^{k}\mu^i\mathfrak{m}^iI^{s-i(n+1)}.$
Since for any $a\in L$, we know that $a\mu\in I^{n+1}$, we get
$L^i\mu^i\subset I^{i(n+1)}.$ By above remark, we get $$I^{(s)}\cap
\mathfrak{m}^{2s}\subset\sum\limits_{i=0}^{k}\sum\limits_{t=0}^{i}\mu^{i}L^tK^{i-t}I^{s-i(n+1)}\subset
\sum\limits_{i=0}^{k}\sum\limits_{t=0}^{i}\mu^{i-t}K^{i-t}I^{s-(i-t)(n+1)}=\sum\limits_{i=0}^{k}\mu^iK^iI^{s-i(n+1)}.$$
Since each term of the summation on the right hand side is naturally
contained in the left hand side, the reverse inclusion follows easily.
\end{proof}

We now define an ordering, called \emph{edgelex} ordering, among the
monomial generators of $I(G)^s$ and $\m^r I(G)^s$ following
\cite[Discussion 4.1]{banerjee}. This helps us in understand certain
colon ideals which are crucial in the study of regularity of powers. 

\begin{definition}\label{ordering}{\rm
Let $G$ be a graph with $E(G)=\{e_1,\dots, e_r\}$ and $I$ be its
edge ideal.  For $A, B\in G(I^s)$, we say that $A>_{\rm edgelex} B$ if
there exists an expression $A=e_{i_1}^{a_1}\cdots e_{i_r}^{a_r}$  such
that  for all expressions $e_{i_1}^{b_1}\cdots e_{i_r}^{b_r} = B$ , we
have $(a_1,\dots, a_r)>_{\rm lex}(b_1,\dots, b_r)$.

Let $J=I^s\mathfrak{m}^r$. Then for any $u, v\in G(J)$, we say that
$u>v$ if there exists an expression $u=fu'$ such that for any
expression of $v=gv'$ with $g \in G(I^s)$ and $v' \in \m^r$, we have
either $f>_{edgelex} g$ or $f=g$ and $u'>_{lex }v'$. Further, we say
$u=fu'$ is a \emph{maximal expression} of $u$ if for any other
expression $f_1u_1 = u$ with $f_1 \in G(I^s)$ and $u_1 \in \m^r$, we
have $f>_{edgelex} f_1$.
}\end{definition}

We now recall the concept of edge-division given in
\cite[Definition 4.2]{banerjee}.
Let $G$ be a graph with $E(G) = \{e_1,\dots, e_r\}$ and $I$ be its
edge ideal of $G$. Let $u\in I^s$. Then for some $j$, we say that
$e_j$ \textit{edge-divides} $u$ if there exists $v\in I^{s-1}$ such
that $u=e_jv$. We denote this by $e_j\mid^{edge}u$.

For example, if $G = C_5$ and $I(G) = (x_1x_2, x_2x_3, x_3x_4, x_4x_5,
x_1x_5) \subset \sk[x_1, \ldots, x_5]$, then $(x_4x_5)^2 >_{edgelex}
(x_1x_5)^2$. Note that with respect to the lex order, the inequality
is reverse. Also, $x_1x_2 \mid^{edge} x_1x_2^2x_3$ and $x_2x_3 \mid
x_1x_2x_3x_4$. Note that the second one is a normal division, not an
edge-division.

Most of the proofs that we do are by some type of induction.
Understanding the behavior of the colon ideal is necessary to apply
induction. We first generalize \cite[Lemma 4.11]{banerjee}.

\begin{lemma}\label{orderLem}
Let $G$ be a graph, $I$ be its edge ideal and $\m$ be the homogeneous
maximal ideal in the appropriate polynomial ring. Let $J=I^s\m^r$.
Then there exists an ordering on minimal monomial generators of $J=(
u_1,\dots, u_m)$ such that for $j< k,$  either
$(u_j:u_k)\subset I^{s+1}:u_k$ or there exists $i<k$ such that $(u_i:u_k)$
is an ideal generated by a variable and it contains $(u_j:u_k)$.
\end{lemma}
\begin{proof}
Consider the ordering on $G(J)$ given in Definition \ref{ordering}.
We prove the result by using induction on $(s,r)$. For $j<k$, let
$u_j=f_1v_1$ and $u_k=f_2v_2$ be maximal expressions, where
$f_1,f_2\in I^s$ and $v_1,v_2\in \m^r$.
	
If $r = 0$, then the assertion follows from \cite[Lemma
4.11]{banerjee}. In particular, if $(s,r)=(1,0)$, then the assertion
holds true.  Assume by induction that the assertion is true for all
$(s_1,r_1) <_{lex}(s,r)$. 
	
Let $ab$ be the maximal edge such that $ab \mid^{edge} f_1$. If
$ab\mid^{edge} f_2$, then write $f_1 = ab f_1'$ and $f_2 = ab f_2'$
for some $f_1', f_2' \in I^{s-1}$. Then $u_j' = u_j/ab = f_1'v_1$ and
$u_k' = u_k/ab = f_2'v_2$ are in $I^{s-1}\m^r$. Moreover, $\left(u_j'
:u_k'\right) = (u_j : u_k)$. By induction, either $(u_j' : u_k')
\subseteq (I^s : u_k')$ or there exists and $i < k$ such that $(u_i' :
u_k')$ is an ideal generated by variables and it contains $(u_j' :
u_k')$. If $(u_j' : u_k') \subseteq (I^s : u_k')$, then clearly $(u_j
: u_k) \subseteq (I^{s+1} : u_k)$. Suppose there exists an $i < k$
such that $(u_i' : u_k')$ is generated by a variable. Clearly $(abu_i'
: u_k) = (u_i' : u_k')$. Hence it is enough to show that $abu_i' >
u_k$ and set $u_i = abu_i'$. But this is obvious since $u_i' \in
G(I^{s-1}\m^r)$, $ab$ is an edge and $ab$ edge divides $f_2$.

Now we assume that $ab \nmid^{edge} f_2$.  If $\gcd(ab, u_k)=1$, then
$(u_j:u_k)\subset ( ab) \subset (I^{s+1} : u_k).$ Hence
the assertion follows.  Suppose $\gcd(ab, u_k) \neq 1$.  Consider the
case when $a\mid u_k$. If $a\mid v_2$, then we claim that $b\nmid u_k$. 
Suppose $b\mid u_k$. If $b \mid v_2$, then we can write $u_k =
v_2'f_2'$, where $f_2' = \frac{ab}{e_j}f_2$ for some edge $e_j$ with
$e_j \mid^{edge} f_2$ and $v_2' = \frac{e_j}{ab}v_2$. Since $f_2'
>_{edgelex} f_2$, the expression $u_k = v_2f_2$ is not maximal
which is a contradiction. Hence $b \nmid v_2$ so that $b | f_2$.
Let $b'$ be such that
$bb'\mid^{edge}f_2$. This implies that $ab\mid^{edge}\dfrac{af_2}{b'}$
and $\dfrac{af_2}{b'}\mid u_k$. Note that
$\dfrac{af_2}{b'}>_{edgelex}f_2$, and hence $f_2v_2$ is not a maximal
expression which is a contradiction to our assumption.
This implies that $b \nmid u_k$ and $(u_j:u_k)\subset ( b)\subset I^{s+1}:u_k$.
If $a \mid f_2$, then, as in the earlier case, we get $b \nmid v_2$.
Then there exists a vertex $c$ such that $ac |^{edge} f_2$. 
Suppose $(u_j : u_k) \subset (b)$. Write $f_1=abf_1'$
and $f_2=acf_2'$ and take $u_i=abf_2'v_2$. Hence $u_i>u_k$ and
$(u_i:u_k) = (b)\supset (u_j:u_k)$. 
Suppose $(u_j:u_k)\not\subset (b)$. Since $b \mid u_j, ~b \mid u_k$.
Also, $b \nmid v_2$. Therefore, $b \mid f_2$ and there exists a vertex
$d$ such that $bd\mid^{edge}f_2$.
If $(u_j : u_k) \subseteq (a)$, then by the symmetry of arguments, we
get $(u_i : u_k) = (a) \supset (u_j : u_k)$, where $u_i =
ab\dfrac{f_2}{bd}v_2$. Hence, for the rest of the proof we may assume
that neither $a$ nor $b$ divides $(u_j : u_k)$.

Let $(u_j:u_k)=(w)$. If $\gcd(f_1,w)=1$, then $w\mid v_1$.
Let $w=xw'$, where $x$ is a variable, and take $w_1$ such that
$w_1\mid \dfrac{u_k}{\gcd(u_j,u_k)}$ with
$\deg(w_1)=\deg(w')$. Set $u_i =\dfrac{u_jw_1}{w'}$. Since
$f_1>_{edgelex}f_2$, we have $u_i>u_k$, and $(u_i:u_k)=( x)$ which contains $w$.
	
Suppose $\gcd(f_1, w) \neq 1$. Let $x$ be a vertex such that $x\mid w $ and $x\mid
f_1$. Note that $x \neq a$. Since $x \mid f_1$, there exists $y$ such that $xy\mid^{edge}f_1$.
If $y$ does not divide $u_k$, then $(u_j:u_k)\subset (xy)$. Since $xy$
is an edge, this implies that $xy u_k \in I^{s+1}$, i.e., $xy \in
I^{s+1} : u_k$. Hence $(u_j : u_k) \subset I^{s+1} : u_k$. Now assume
that $y \mid u_k$. If $y\mid v_2$, then $xv_2f_2 \in I^{s+1}$, since
$xy$ is an edge and $f_2 \in I^s$. Therefore, $x \in I^{s+1} : u_k$. Hence
$(u_j:u_k)\subset ( x ) \subseteq I^{s+1}:u_k$.  If $y\mid f_2$, then
there exists $z$ such that $yz\mid^{edge}f_2$.  Write $f_1=abxyf_1''$
and $f_2=acyzf_2''$ for some $f_1'', f_2'' \in I^{s-2}$.  Since
$(u_j:u_k)=(w)$, we get $u_j \mid wu_k$, and hence $abf_1''v_1\mid
w'zacf_2''v_2$. This implies that $(w'z)\subset (abf_1''v_1:
acf_2''v_2)$.  Let $(abf_1''v_1: acf_2''v_2) = (w_1')$. This gives us
$abf_1''v_1 \mid w_1acf_2''v_2$, and hence $u_j\mid w_1xu_k$ which
forces that $w'x\mid w_1x$. This implies that $(abf_1''v_1:
acf_2''v_2)$ is equal either to $(w')$ or to $(w'z)$.  Note
that $abf_1''>_{edgelex} acf_2''$. Therefore by induction
$(abf_1''v_1: acf_2''v_2)\subset I^{s}:acf_2''v_2$ or
there exists $u'\in G(I^{s-1}\m^r)$ such that $(u':acf_2''v_2)$ is
generated by a variable and it contains $(abf_1''v_1: acf_2''v_2)$.
	 
Assume that $(abf_1''v_1: acf_2''v_2)= (w')\supset (u_j:u_k)$. Suppose
$(abf_1''v_1: acf_2''v_2)\subset I^{s}:acf_2''v_2$. This implies
that $(u_j:u_k)\subset(abf_1''v_1: acf_2''v_2)\subset I^{s+1}: u_k$.
Suppose there exists $u' \in G(I^{s-1}\m^r)$ such that 
$(u':acf_2''v_2)=(l)$ for some variable $l$
which divides $w'$. Therefore, by taking $u_i = yzu'$, we get
$(u_i:u_k)=(u':acf_2''v_2)=(l)$ which divides $w'$, and hence $w$.
	 
Suppose $(abf_1''v_1: acf_2''v_2)= (w'z)$. If $(abf_1''v_1:
acf_2''v_2)\subset I^{s}:acf_2''v_2$, i.e., $w'zacf_2''v_2\in I^s$,
then $w'zxyacf_2''v_2\in I^{s+1}$, i.e., $wu_k\in I^{s+1}$. 
Suppose there exists $u' \in G(I^{s-1}\m^r)$ such that
$(u':acf_2''v_2)=(l)$, where $l$ is a variable and $l\mid w'z$. If
$l=z$,  then take $u_i =xyu'$. This gives us $(u_i :u_k)=(x)$. If
$l\neq z$, then take $u_i =yzu'$. Then we get $(u_i :u_k)=(l)$. In
both cases, $(u_i :u_k)$ is generated by a variable and it contains $(u_j:u_k)$ which completes the proof.
\end{proof}

We now recall the definition of even-connection introduced by Banerjee
in \cite{banerjee}.

\begin{definition}{\rm
	Let $G$ be a graph and $x$ and $y$ be vertices of $G$. Then we say that $x$ and $y$ are \emph{even connected} with respect to $u=e_1\cdots e_s$ if there is a path $p_0p_1\cdots p_{2k+1}$, $k\geq 1$ in $G$ such that 
	\begin{enumerate}[i)]
		\item $p_0=x$ and $p_{2k+1}=y$.
		\item For all $1\leq l\leq k$, we have $p_{2l-1}p_{2l}=e_i$ for some $i$.
		\item For all $i$, we have $|\{l\geq0:p_{2l-1}p_{2l}=e_i \}|\leq|\{j:e_j=e_i\}|.$
	\end{enumerate}
}\end{definition}
One of the most important property of the even connection is that it
describes the generators of the colon ideal $I^s : u$.
\begin{theorem}\cite[Theorem 6.7]{banerjee}\label{evenConThm}
Let $G$ be a graph and $I$ be its edge ideal. Let $u\in G(I^{s-1})$.
Then $I^s:u=I+( xy: x \text{ is even connected to } y \text{ with respect to } u).$
\end{theorem}

We further analyze the even-connected edges in
this class of edge ideals and certain colon ideals which come up in
the induction step.
\begin{lemma}\label{leafOrderLem}
Let $G$ be a graph obtained by taking the clique-sum along the
vertices or edges of an odd cycle $C_{2n+1}$ and some bipartite
graphs. Let $\{z_1, \ldots, z_m\} = V(G) \setminus N_G(C_{2n+1})$. 
Assume that $z_i$ is not part of any cycle for all $i = 1, \ldots,
m$.  Then there exists an ordering on $G(I^s)=\{
u_1,\dots, u_r\} $ such that if $z_i$ and $z_j$ are
even-connected with respect to $u_t$ for some $1 \leq t \leq r$, then
there exists $u_s> u_t$ such
that $(u_s:u_t)=(z_k)$, where $k = \min \{i,j\}$.   
\end{lemma}

\begin{proof}
Since $z_i$ is not part of any cycle, it follows that the induced
subgraph on $V(G) \setminus \ N_G(C_{2n+1})$ is a forest. After a
re-ordering of the vertices, assume that
$e_1$ is a leaf in $G$ having pendant vertex $z_{1}$ and $e_i$ is a
leaf in $G \setminus \{e_1, \ldots, e_{i-1}\}$ with pendant vertex
$z_i$, for $i = 2, \ldots, m$.  Set $z_1 > \cdots > z_m$, $e_1 > \cdots > e_m$ and on $E(G)
\setminus \{e_1, \ldots, e_m\}$, set the lexicographic ordering with
$y_1 > \cdots > y_l > x_1 > \cdots > x_{2n+1}$ and such
that for any $e \in E(G) \setminus \{e_1, \ldots, e_m\}$, $e_m > e$.
Now, take the edgelex ordering on $I^s$.
	
Suppose $z_i$ and $z_j$ are even
connected with respect to $u_l =e_{i_1}\cdots e_{i_s}$. Without loss
of generality, we may assume that $i < j$. Hence $z_j < z_i$ and $e_j < e_i$. Let 
$z_ip_1\cdots p_{2k}z_j$, $k\geq 1$ be an even-connection in $G$. 

We claim that $z_ip_1>p_1p_2$. If $z_ip_1 < p_1p_2$, then $z_i < p_2$.
This implies that $p_2 = z_{i_1}$ for some $i_1 < i$. Since $z_{i_1}$
is obtained as a pendant vertex after removing $z_1, \ldots,
z_{i_1-1}$ and both $z_i$ and $p_1$ are less than $z_{i_1}$, $p_3 =
z_{i_2}$ for some $i_2 < i_1$. Continuing like this, we obtain that
$p_{2k+1} = z_j > z_i$ which is a contradiction to our assumption that
$i < j$. Hence $z_ip_1 > p_1p_2$.  
Set $u_s = z_ip_1\dfrac{u_l}{p_1p_2}$. Then $u_s >u_t$ and $(u_s :
u_t) = z_i$.
\end{proof}

\begin{remark}\label{remark}{\rm
Let $f=\mu^igu\in G(\mu^iK^iI^{s-i(n+1)})$ and $M=\supp(g)$. Then we
have the following:
\begin{enumerate}[i)]
  \item Let $l\in M$ and $l'\in N_G(l)$. This implies
	$\dfrac{\mu}{x_j}ll'u\in I^{s-(i-1)(n+1)}$ for any $j$. Hence
	$l'\mu^igu = \dfrac{g}{l}ll'\mu^iu \in
	\mu^{i-1}K^{i-1}I^{s-(i-1)(n+1)}$ which shows that $N_G(M)\subset
	\mu^{i-1}K^{i-1}I^{s-(i-1)(n+1)}:f$.
  \item Let $l\in M \cup V(C_n)$ and $l'\in V(G)$ such that $l$ and $l'$ is an
	even connection with respect to $\dfrac{\mu u}{x_j}$ for some $j$.
	Hence $\dfrac{\mu}{x_j}ll'u\in I^{s-(i-1)(n+1)}$ for some $j$.
	Hence $l'\mu^igu\in \mu^{i-1}K^{i-1}I^{s-(i-1)(n+1)}$ which shows
	that $l'\in \mu^{i-1}K^{i-1}I^{s-(i-1)(n+1)}:f$.
\end{enumerate}
}\end{remark}

To understand the colon with symbolic power, we study
the colon with ideals in the decomposition of the symbolic power.
\begin{lemma}\label{colonLemma}
Let $G$ be as in Lemma \ref{leafOrderLem} and $f=\mu^igu\in
G(\mu^iK^iI^{s-i(n+1)})$ for $1\leq i\leq k$ with $f\notin
\mu^{i-1}K^{i-1}I^{s-(i-1)(n+1)}$, where $u\in I^{s-i(n+1)}$. Then $$
\mu^{i-1}K^{i-1}I^{s-(i-1)(n+1)}:f=I + L',$$
where $L'$ is an ideal containing $L$ and generated by a set of
variables.
\end{lemma}
\begin{proof}
Note that for any $a \in L$, we know that $a\mu\in I^{n+1}$, and
hence we get  $$L\subset \mu^{i-1}K^{i-1}I^{s-(i-1)(n+1)}:f.$$ 
	
We first claim
that $I^{s-(i-1)(n+1)} : \mu u$ is generated in degree at most $2$. 
Since $\dfrac{\mu u}{x_i}\in I^{s-(i-1)(n+1)-1}$, the ideal
$I^{s-(i-1)(n+1)} :\dfrac{\mu u}{x_i}$ is of the form $I^{t+1} : e_1
\cdots e_t$ which is generated in degree $2$.
Let $v\in I^{s-(i-1)(n+1)} : \mu u$. This implies that
$vx_i\in I^{s-(i-1)(n+1)} :\dfrac{\mu u}{x_i}$. Thus there exists a
monomial $v'$ of degree $2$ such that $v'\mid vx_i$.  If $x_i \nmid
v'$, then $v' \mid v$. If $x_i\mid v'$, then $v'/x_i\in
I^{s-(i-1)(n+1)} : \mu u$. Hence $I^{s-(i-1)(n+1)} : \mu u$ is
generated in at most degree 2. Suppose $v \in G(I^{s-(i-1)(n+1)} : \mu
u)$ such that $v \mid g$. Suppose $\deg(v) = 1$. Then $f =
\mu^{i-1}\dfrac{g}{v}v\mu u \in \mu^{i-1}K^{i-1}I^{s-(i-1)(n+1)}$
which is a contradiction to our assumption that $f$ is not in that
ideal. Hence $\deg(v) = 2$. Then one can see as
above that $f \in \mu^{i-1}K^{i-2}I^{s-(i-1)(n+1)}$. Hence $K
\subseteq \mu^{i-1}K^{i-1}I^{s-(i-1)(n+1)} : f$ so that
$\mu^{i-1}K^{i-1}I^{s-(i-1)(n+1)} : f = \m = L' = I +
L'$. For the rest of the proof, we may assume that if $v \in
G(I^{s-(i-1)(n+1)} : \mu u)$, then $v \nmid g$.

Now, let $K^{i-1}=(g_1,\dots, g_k,\dots, g_r)$ with $g_j \mid
g$ for $ j=1,\dots, k$. Suppose $g=l_jg_j$ for $j=1, \dots, k$.
Note that  
\[
  \mu^{i-1}K^{i-1}I^{s-(i-1)(n+1)}:f=\sum_{j=1}^{k}I^{s-(i-1)(n+1)}:\mu
  l_ju+\sum_{j=k+1}^{r}g_jI^{s-(i-1)(n+1)}:\mu gu.\]
We claim that for $k+1\leq j\leq r$, $g_jI^{s-(i-1)(n+1)}:\mu
gu\subset \sum_{j=1}^{k}I^{s-(i-1)(n+1)}:\mu l_ju$. Suppose that
$\gcd(g_j,g)=h_j$. Write $g_j=h_jg_j'$ and $g=h_jg'$. 
Now, let $a$ be
a monomial such that $a\mu g'u\in I^{s-(i-1)(n+1)}$. 
Hence $ag' \in
I^{s-(i-1)(n+1)} : \mu u$ and this colon ideal is generated in at most
degree $2$, where the degree $2$ generators are either edges or even
connections. Then there exists a monomial generator $v$ of
$I^{s-(i-1)(n+1)} : \mu u$  dividing $ag'$. 
If $a \mu u \in I^{s-(i-1)(n+1)}$, then we are through. 
Assume that $a\mu u\notin I^{s-(i-1)(n+1)}$. If $v \mid a$, then $v
\mu u$ and hence $a
\mu u$ belongs to $I^{s-(i-1)(n+1)}$ which is a contradiction to our
assumption. Hence $v \nmid a$. Also, $v \nmid g'$ (since $v \nmid g$).
Hence we may write $v=l_jv'$ such that $l_j\mid g'$
and $v'\mid a$. This implies that $v'\in I^{s-(i-1)(n+1)}:l_j\mu u$,
and hence $a\in I^{s-(i-1)(n+1)}:l_j\mu u$.
Therefore $a\in
\sum_{j=1}^{k}I^{s-(i-1)(n+1)}:\mu l_ju$ which proves the claim. 

Now we claim that $I^{s-(i-1)(n+1)}:\mu l_ju = I + I' + L_j$, where $I'$ is
the ideal generated by the even connections with respect to $\frac{\mu u}{x_a}$ for all $a$. Let $v\in I^{s-(i-1)(n+1)}:\mu l_ju$. Then $vx_al_j\in I^{s-(i-1)(n+1)}:\frac{\mu u}{x_a}$. 
As  $  I^{s-(i-1)(n+1)}:\frac{\mu u}{x_a}$ is generated by edges and
even connections with respect to $\frac{\mu u}{x_a}$, there exists
$w = w_1w_2$ which is an edge or an even connection with respect to $\frac{\mu
u}{x_a}$ such that $w \mid vx_al_j$. If $w\mid v$, then we are done.
If $w \mid x_al_j$, then this implies that $\mu l_j u\in I^{s-(i-1)(n+1)}$, and hence $f\in
\mu^{i-1}K^{i-1}I^{s-(i-1)(n+1)}$ which is a contradiction. Now, let
$w\nmid v$ and $w\nmid x_al_j$. We may assume that $w_1 \mid v$ and
$w_2 \mid x_al_j$. If $w_2 = x_a$, then $w_2 \in N_G(C_n) = L$. If
$w_2 = l_j$, then $w_1 \in N_G(M) \subset L'$, by Remark
\ref{remark}(i). Hence, In either case, $w\in L'$. Hence we get that $
\mu^{i-1}K^{i-1}I^{s-(i-1)(n+1)}:f=I + L',$ where $L'=\sum_{j=1}^k L_j$. Now, Using Lemma \ref{leafOrderLem}, we know that $I'\subset L'$ which completes the proof.
%
\end{proof}

In the process of understanding colon with symbolic power, in a
step-by-step manner, we now study the colon with respect to the
partial sums in the decomposition of symbolic powers.
\begin{lemma}\label{ordLem}
Let $G, \mu, K, L$ be as defined in the beginning of the section.
Assume that $z_r$ is not part of any cycle for all $r = 1, \ldots,
m$.  Let $I = I(G)$  and for $1 \leq i \leq \lfloor \frac{s}{n+1}
\rfloor+1$, set
$I_{i-1}=\sum\limits_{t=0}^{i-1}\mu^tK^tI^{s-t(n+1)}$. Then there exists
an ordering of $G\left(\mu^iK^iI^{s-i(n+1)}\right)=\{u_1,\dots, u_r\}$
such that for all $j = 0, \ldots, r-1$, $$\left(I_{i-1}+( u_1,\dots, u_j) \right):
u_{j+1}=I + L'',$$ where $L''$ is an ideal containing $L$ and generated
by a subset of variables.
\end{lemma}

\begin{proof} Let $u_{j+1}= \mu^ifu$, where $f \in G(K^i)$ and $u \in
G(I^{s-i(n+1)})$. In order to prove the assertion we claim that for a
fixed $i$ and $t<i-1$, if $\mu^igu\notin \mu^tK^tI^{s-i(n+1)}$, then
$$\mu^tK^tI^{s-t(n+1)}:\mu^igu\subset
\mu^{i-1}K^{i-1}I^{s-(i-1)(n+1)}:\mu^igu.$$
We first consider the linear part of the left hand side colon ideal
and show that it is contained in the right hand side. Note that 
\[K^t I^{s-t(n+1)} : \mu^{i-t}gu   = \sum_{j = 1}^k I^{s-t(n+1)} :
\mu^{i-t}\frac{g}{g_j} u + \sum_{j = k+1}^r \frac{g_j}{\gcd(g, g_j)}I^{s-t(n+1)} :
\mu^{i-t}\frac{g}{\gcd(g,g_j)} u,\] 
where $g_j \in K^t$ is a
divisor of $g$ for $1 \leq j \leq k$ and for $k+1 \leq j \leq r$, $g_j
\in K^t$ that does not divide $g$. As in the case of proof Lemma
\ref{colonLemma}, it can be shown that the second term in the above
summation is contained in the first. Hence, to prove the assertion, it
is enough to consider the first summation.

First of all, note that $\mu^{i-t}u \in
I^{s-t(n+1)-\left[\frac{i-t+1}{2}\right]}$. Again, as the proof of
Lemma \ref{colonLemma}, one can see that this ideal is generated in
degree at most $2$, with the degree $2$ part generated by some of the
edges and even connections. Hence, if $l\in V(G)$ is such that $l\in
K^tI^{s-t(n+1)-\left[\frac{i-t-1}{2}\right]}:\mu^{i-t}gu$, then there
exists $l'\in V(C_{2n+1})$ or $l'\in \supp(g)$ such that $ll'\in
E(G)$ or an even connection. By Remark \ref{remark}, we get that
$l\in\mu^{i-1}K^{i-1}I^{s-(i-1)(n+1)}:\mu^igu.$ Now, note that $u\in
I^{s-i(n+1)}$ and $\mu^{i-t}\in
I^{(i-t)(n+1)-\left[\frac{i-t+1}{2}\right]}$. Hence by Lemma 4.3, we
know that
$K^tI^{s-t(n+1)-\left[\frac{i-t-1}{2}\right]}:\mu^{i-t}gu=I+L_t'$,
where $L_t'$ is generated by a set of variables and contains $L$. This
implies that $\mu^tK^tI^{s-t(n+1)}:\mu^igu\subset
I+L'=\mu^{i-1}K^{i-1}I^{s-(i-1)(n+1)}:\mu^igu.$ This proves the claim.

Thus by Lemma  \ref{colonLemma}, we
have $I_{i-1}:u_{j+1}= (I^{s-i(n+1)+1}:u)+L'$. On
$G(\mu^iK^iI^{s-i(n+1)}),$ define an ordering induced by the ordering in
Lemma \ref{leafOrderLem} and Definition \ref{ordering}.
By Lemma \ref{orderLem}, there exists a largest ideal generated by a
subset of variables, say $L_2$, such that $L_2 \subset (u_1, \ldots,
u_j) : u_{j+1}$ and $( u_1,\dots, u_j): u_{j+1}\subset
I^{s-i(n+1)+1}:u +L_2$. Take $L'' = L' + L_2$. Then it follows from 
Theorem \ref{evenConThm} and Lemma \ref{leafOrderLem} that
$\left(I_{i-1}+( u_1,\dots, u_j) \right):
u_{j+1}\subseteq I + L''.$ Since $L_2 \subset (u_1, \ldots, u_j) :
u_{j+1}$ and $I_{i-1} : u_{j+1} = (I^{s - i(n+1)+1} : u) + L' \supset
I + L'$, we get the reverse containment as well.
\end{proof}

\begin{remark}\label{regRem}{\rm
Let $H$ be the induced subgraph on $V(G) \setminus N_G(C_{2n+1})$. By Lemma \ref{ordLem}, we
know that $$(I_{i-1}+( u_1,\dots, u_{j})):u_{j+1}=I+L'',$$
where $L\subset L''\subset \mathfrak{m}$. This implies that $I+L''$
corresponds to an induced subgraph of $H$. Therefore, we get
$$\reg\left(\dfrac{S}{(I_{i-1}+( u_1,\dots,
u_{j})):u_{j+1}}\right)\leq \reg\left(\dfrac{S}{I(H)}\right) = \nu(H).$$
}\end{remark}

\begin{proposition}\label{regLem}
Let the notation be as in Lemma \ref{ordLem} and $H$ denote the
induced subgraph on $V(G) \setminus N_G(C_{2n+1})$.
If $\nu(G)-\nu(H)\geq 3$, then for $1\leq i\leq \lfloor \frac{s}{n+1}
\rfloor+ 1$,  $\reg(I^s)=\reg(I_{i-1})$.
\end{proposition}

\begin{proof}
Result is true for $i=1$. Assume that it is true for $i-1$. Using
Lemma \ref{ordLem}, write $G(\mu^iK^iI^{s-i(n+1)})=\{u_1,\dots, u_r\}$
such that for all $j = 0, \ldots, r-1$, $$I_i = I_{i-1}+( u_1,\dots,
u_{r}) \text{ and } (I_{i-1}+( u_1,\dots, u_{j})):u_{j+1}=I+L''.$$
	
For $j=0$, consider the following exact sequence
$$ 0\longrightarrow
\dfrac{S}{I+L''}(-2s)\overset{\cdot u_1}{\longrightarrow}
\dfrac{S}{I_{i-1}}\longrightarrow\dfrac{S}{I_{i-1}+( u_1)}\longrightarrow 0.$$

From Remark \ref{regRem}, we know that
$$\reg\left(\dfrac{S}{I+L''}(-2s)\right)\leq 2s+\nu(H)\leq
2s+\nu(G)-3<\reg\left(\dfrac{S}{I^s}\right)=\reg\left(\dfrac{S}{I_{i-1}}\right),$$
where the third inequality follows from \cite[Theorem 4.5]{BHT}.
Hence
$\reg\left(\dfrac{S}{I_{i-1}}\right)=\reg\left(\dfrac{S}{I_{i-1}+(
u_1)}\right)$. Assume by induction on $j$ that $\reg(I_{i-1} + (u_1,
\ldots, u_{j-1})) = \reg(I^s)$. Since $(I_{i-1} + (u_1, \ldots, u_j))
: u_{j+1} = I + L''$, we get the desired equality from the short exact
sequence:
$$ 0\longrightarrow
\dfrac{S}{I+L''}(-2s)\overset{\cdot u_{j+1}}{\longrightarrow}
\dfrac{S}{I_{i-1}+(u_1, \ldots,
u_j)}\longrightarrow\dfrac{S}{I_{i-1}+(u_1, \ldots, u_{j+1} )}\longrightarrow 0.$$

\end{proof}
We are now ready to prove our second main theorem.
\begin{theorem}\label{mainTheorem}
Let $G$ be a graph obtained by taking clique sum of a $C_{2n+1}$ and
some bipartite graphs. Let $H$ be an
induced subgraph of $G$ on vertices $V\setminus\bigcup_{x\in
V(C_{2n+1})}N_G(x)$. Assume that none of the vertices of $H$ is part
of any cycle in $G$. If $\nu(G)-\nu(H)\geq 3$, then
$\reg\left(I^{(s)}\right)=\reg\left(I^s\right)$.
\end{theorem}

\begin{proof} Let $s \geq 1$ and $k = \lfloor \frac{s}{n+1} \rfloor$.
Consider the following exact sequence
$$0\longrightarrow \dfrac{S}{I_k}\longrightarrow
\dfrac{S}{I^{(s)}} \oplus \dfrac{S}{\mathfrak{m}^{2s}} \longrightarrow
\dfrac{S}{I^{(s)}+ \mathfrak{m}^{2s}} \longrightarrow0,$$ where $I_k =
\sum_{t=0}^k \mu^tK^tI^{s-t(n+1)}$.
Since $\nu(G)\geq 2$, we have
$\reg\left(\dfrac{S}{I^{(s)}}\right)>\reg\left(\dfrac{S}{\mathfrak{m}^{2s}}\right)=\reg\left(\dfrac{S}{I^{(s)}+\mathfrak{m}^{2s}}\right)$.
Hence
$\reg\left(\dfrac{S}{I_k}\right)=\reg\left(\dfrac{S}{I^{(s)}}\right)$.
Since $\nu(G)-\nu(H)\geq 3$,  by Proposition \ref{regLem}, we get that $\reg(I^s)=\reg(I^{(s)}).$
\end{proof}

\begin{remark}{\rm
If the odd cycle in $G$ is of length at least $9$, then the condition 
$\nu(G) - \nu(H) \geq 3$ is always satisfied. 
\begin{enumerate}
  \item If the unique odd cycle in $G$ is of length $7$, then the
	hypothesis of Theorem \ref{mainTheorem} is satisfied if a $P_3$ is
	attached to $C_7$. 
  \item If the unique odd cycle in $G$ is of length $5$, then the
	hypothesis of Theorem \ref{mainTheorem} is satisfied if either two
	$P_3$'s are attached to a single vertex or a $P_3$ and a $P_2$ are
	attached to adjacent vertices (see figure below).
  \item If the unique odd cycle in $G$ is of length $3$, then the
	hypothesis of Theorem \ref{mainTheorem} is satisfied if either two
	$P_3$'s are attached to a single vertex or on each vertex of
	$C_3$ a $P_3$ is attached (see figure below).
  \item It may also be noted that the class of graphs considered in
	Theorem \ref{mainTheorem} is not a subset of unicyclic graphs. It
	also includes graphs which are obtained by taking clique sum of
	copies of $C_4$ along the edges of an odd cycle (see figure
	below).
\end{enumerate}
}\end{remark}

We illustrate with pictures, some of the graphs for which the
regularity of the symbolic powers of their edge ideals are same as
that of their regular powers.

\vskip 4mm
\begin{center}

\begin{tikzpicture}[scale=0.8]
\draw (1,1)-- (3,1);
\draw (2,2)-- (1,1);
\draw (2,2)-- (3,1);
\draw (2,2)-- (1,3);
\draw (2,2)-- (3,3);
\draw (6,2)-- (5,1);
\draw (7,1)-- (5,1);
\draw (6,2)-- (7,1);
\draw (5,1)-- (4,0);
\draw (6,2)-- (6,3.5);
\draw (7,1)-- (8,0);
\draw (1,4)-- (3,4);
\draw (3,4)-- (3,5);
\draw (3,5)-- (2,6);
\draw (2,6)-- (1,5);
\draw (1,5)-- (1,4);
\draw (2,6)-- (1,7);
\draw (2,6)-- (3,7);
\draw (5,4)-- (7,4);
\draw (7,4)-- (7,5);
\draw (5,4)-- (5,5);
\draw (5,5)-- (6,6);
\draw (6,6)-- (7,5);
\draw (6,6)-- (7,7);
\draw (7,5)-- (8,5);
\draw (10,4)-- (9,3);
\draw (9,2)-- (9,3);
\draw (10,4)-- (11,3);
\draw (11,3)-- (11,2);
\draw (11,1)-- (11,2);
\draw (9,2)-- (9,1);
\draw (9,1)-- (11,1);
\draw (11,3)-- (12,3);
\draw (12,3)-- (12,2);
\draw (12,2)-- (11,2);
\draw (10,4)-- (11,5);
\draw (10,4)-- (9,5);
\draw (12,3)-- (13,4);
\draw (12,2)-- (13,1);
\begin{scriptsize}
  \draw [fill] (1,1) circle (1pt);
\draw [fill] (3,1) circle (1pt);
\draw [fill] (2,2) circle (1pt);
\draw [fill] (1,3) circle (1pt);
\draw [fill] (1.525,2.475) circle (1pt);
\draw [fill] (3,3) circle (1pt);
\draw [fill] (2.535,2.535) circle (1pt);
\draw [fill] (6,2) circle (1pt);
\draw [fill] (5,1) circle (1pt);
\draw [fill] (7,1) circle (1pt);
\draw [fill] (4,0) circle (1pt);
\draw [fill] (6.01,3.46) circle (1pt);
\draw [fill] (8,0) circle (1pt);
\draw [fill] (6.004793826523432,2.6998986724210727) circle (1pt);
\draw [fill] (4.515,0.515) circle (1pt);
\draw [fill] (7.515,0.485) circle (1pt);
\draw [fill] (1,4) circle (1pt);
\draw [fill] (3,4) circle (1pt);
\draw [fill] (3,5) circle (1pt);
\draw [fill] (2,6) circle (1pt);
\draw [fill] (1,5) circle (1pt);
\draw [fill] (1,7) circle (1pt);
\draw [fill] (3,7) circle (1pt);
\draw [fill] (1.465,6.535) circle (1pt);
\draw [fill] (2.535,6.535) circle (1pt);
\draw [fill] (5,4) circle (1pt);
\draw [fill] (7,4) circle (1pt);
\draw [fill] (7,5) circle (1pt);
\draw [fill] (5,5) circle (1pt);
\draw [fill] (6,6) circle (1pt);
\draw [fill] (7,7) circle (1pt);
\draw [fill] (8,5) circle (1pt);
\draw [fill] (6.505,6.505) circle (1pt);
\draw [fill] (10,4) circle (1pt);
\draw [fill] (9,3) circle (1pt);
\draw [fill] (9,2) circle (1pt);
\draw [fill] (11,3) circle (1pt);
\draw [fill] (11,2) circle (1pt);
\draw [fill] (11,1) circle (1pt);
\draw [fill] (9,1) circle (1pt);
\draw [fill] (12,3) circle (1pt);
\draw [fill] (12,2) circle (1pt);
\draw [fill] (11,5) circle (1pt);
\draw [fill] (9,5) circle (1pt);
\draw [fill] (13,4) circle (1pt);
\draw [fill] (13,1) circle (1pt);
\draw [fill] (12.525,3.525) circle (1pt);
\draw [fill] (12.495,1.505) circle (1pt);
\draw [fill] (10.495,4.505) circle (1pt);
\end{scriptsize}
\end{tikzpicture}
\end{center}

\renewcommand{\bibname}{References}
\bibliographystyle{plain}  
\bibliography{refs_reg}
\end{document}